\documentclass[]{article}
\usepackage{proceed2e, latexsym,array,delarray,amsmath,amsthm,amssymb,epsfig,diagrams}
\usepackage{natbib}
\usepackage{hyperref, url}

\title{Identifying Causal Effects with Computer Algebra} 

\author{ {\bf Luis D. Garc\'{\i}a-Puente} \\  
 Dept.~of Math.~and Stats. \\  
 Sam Houston State University\\ 
 Huntsville, TX 77341 \\ 
 \href{mailto:lgarcia@shsu.edu}{lgarcia@shsu.edu}\\
 \And 
 {\bf Sarah Spielvogel}  \\ 
  Dept.~of Math.~and Stats.  \\ 
 Sam Houston State University\\ 
 Huntsville, TX 77341 \\ 
 \href{mailto:sxs013@shsu.edu}{sxs013@shsu.edu}\\
 \And 
 {\bf Seth Sullivant}   \\ 
 Department of Mathematics \\          
 North Carolina State University    \\           
Raleigh, NC 27695\\ 
 \href{mailto:smsulli2@ncsu.edu}{smsulli2@ncsu.edu}} 

\theoremstyle{plain}
\newtheorem{thm}{Theorem}

\newtheorem{prop}[thm]{Proposition}

\newtheorem*{thm*}{Theorem}
\newtheorem*{lemma*}{Lemma}
\newtheorem*{prop*}{Proposition}
\newtheorem*{cor*}{Corollary}
\newtheorem*{conj*}{Conjecture}

\theoremstyle{definition}

\newtheorem{ex}[thm]{Example}

\theoremstyle{remark}

\newcommand{\nn}{\mathbb{N}}

\newcommand{\rr}{\mathbb{R}}

\newcommand{\bfa}{\mathbf{a}}

\newcommand{\bff}{\mathbf{f}}

\newcommand{\bfp}{\mathbf{p}}

\newcommand{\bft}{\mathbf{t}}
\newcommand{\bfu}{\mathbf{u}}
\newcommand{\bfv}{\mathbf{v}}
\newcommand{\bfw}{\mathbf{w}}

\newcommand{\cg}{\mathcal{G}}
\newcommand{\ci}{\mathcal{I}}

\newcommand{\cn}{\mathcal{N}}
\newcommand{\cp}{\mathcal{P}}

\newcommand{\bi}{\leftrightarrow}

\begin{document} 
 
\maketitle 
 
\begin{abstract} 
The long-standing identification problem for causal effects in graphical models has many partial results but lacks a systematic study.  We show how computer algebra can be used to either prove that a causal effect can be identified, generically identified, or show that the effect is not generically identifiable.   We report on the results of our computations for linear structural equation models, where we determine precisely which causal effects are generically identifiable for all graphs on three and four vertices.
\end{abstract} 

\section{INTRODUCTION}

Consider a parametric statistical model $p_\bullet : \Theta \rightarrow P(X)$, that associates to a parameter vector $\theta \in \Theta$ a probability density $p_\theta(x)$ on the sample space $X$.  Let $s: \Theta \rightarrow \mathbb{R}$ be a parameter of interest.  The identification problem asks:  Does there exist a function $\Phi$ from the model $p_\Theta =  \{ p_\theta : \theta \in \Theta \}$ to $\mathbb{R}$ such that $\Phi \circ p_\theta  =  s(\theta)$ for all $\theta \in \Theta$?  If there is such a function, then the parameter $s(\theta)$ is said to be identifiable (and $\Phi$ is the identification formula), and if there is no such function the parameter is not identifiable.  The main focus of this paper is on generically identifiable parameters (i.e. almost everywhere identifiable), which are identifiable expect possibly on a set of measure zero.

We will describe a general framework using computer algebra for addressing such questions.  Our perspective is that, in most problems of interest in machine learning, both the map $p_\bullet$ that associates a density to a parameter vector, and the parameter of interest $s(\theta)$ are polynomial (or rational) functions of the parameters $\theta$.  In this case, if the parameter is (generically) identifiable, the identification formula must be, at worst, an algebraic function of the probability distribution, and such an algebraic identification formula can be detected or proven not to exist using Gr\"obner basis computations.  Some past work on using computer algebra for the identifiability problem includes \citep{Geiger1999} and \citep{Merckens94}.

Our motivation for laying down this general framework is to provide a systematic study of the 
 identifiability of direct and indirect causal effects in causal graphical models.  We provide a systematic study of the generic identifiability of linear structural equation models (SEMs) using our general framework in Section 4.  This problem has been much studied in the literature of machine learning, graphical models, statistics, econometrics, etc.  There exist many different graphical criteria that guarantee that some particular causal effects can be identified or generically identified including the ``single-door'' and ``instrumental variables'' criteria for direct effects; the ``back-door'' criterion for total effects \citep[see][]{Pearl00}; the ``G-criterion'' \citep{Brito06}; the various criteria introduced by \citet{Tian04, Tian05, Tian09}.  Other references among many include: \citep{Fisher66}, \citep{Kuroki99}, \citep{Robins87}, and \citep{Simon53}.

\citet{Tian02} gave an algorithm, proven complete by \citet{Shpitser06} and \citet{Huang06} for identification of parameters in non-parametric structural equation models (that is, an algorithm which decided identifiability depending only on the type of graph).  While there are many conditions that exist for generic identifiability, there is no known necessary and sufficient condition to decide generic identifiability in either the general nonparametric case, or in specific situations (e.g.~linear SEMs).     One goal in this paper is to provide a systematic study to classify the linear SEMs on small numbers of variables whose parameters are generically identifiable.

In the next section, we describe the general algebraic framework for performing identifiability computations.  In Section 3, we describe the problem for Gaussian structural equation models, and give examples of code that shows how to perform the computations from Section 2 for these models.  In Section 4 we report on the results of our computations.

\section{COMPUTER ALGEBRA FOR PARAMETER IDENTIFICATION}\label{section:comp-alg}

In this section, we describe the general framework for addressing identifiability problems using computer algebra.  We refer the reader to \citep{Cox07} for background on computer algebra, ideals, and related topics which are used in this section.

Let $\Theta \subseteq \mathbb{R}^d$ be a full dimensional parameter set.  In most applications $\Theta$ is a convex subset of $\rr^d$.  Let $\rr[\bft]:= \rr[t_1,\ldots, t_d]$ denote the set of all polynomials in the indeterminates (i.e. polynomial variables) $t_1, t_2, \ldots, t_d$ with real coefficients.  The set $\rr[\bft]$ is called the polynomial ring, a \emph{ring} being an algebraic structure with compatible addition and multiplication operations.  Let $f_1, \ldots, f_n \in \rr[\bft]$ be polynomials.  These polynomials define a function $\bff : \Theta \rightarrow \rr^n$ by $\bff(\theta) = (f_1(\theta),\ldots, f_n(\theta))^T$.  The image of $\bff$ is the set $\bff(\Theta) :=  \{\bff(\theta) : \theta \in \Theta \}$.

We define a parameter to be a polynomial function $s : \Theta \rightarrow \rr$ which is not constant on $\Theta$.  
The parameter $s$ is identifiable if there exists a map $\Phi: \rr^n \rightarrow \rr$ such that $s(\theta) =  \Phi \circ \bff(\theta)$ for all $\theta \in \Theta$.  Note the use of the word map, rather than function-- we do not require that $\Phi$ be defined on all of $\rr^n$, but only on $\bff(\Theta)$.  This leads to our next definition.  The parameter $s$ is \emph{generically identifiable} if there exists a map $\Phi: \rr^n \rightarrow \rr$ and a dense open subset $U$ of $\Theta$ such that $s(\theta) =  \Phi \circ \bff(\theta)$ for all $\theta \in U$.  Generic identifiability is also called almost everywhere identifiability in the literature.

An important special case concerns the coordinate functions $s(\theta) = \theta_i$.  If all these functions are (generically) identifiable there exists a (generic) inverse map to the function $\bff$, in which case every parameter is (generically) identifiable.

\begin{ex} \label{ex:instrument}
Let $\theta =  (\omega_{11}, \omega_{22}, \omega_{23}, \omega_{33}, \lambda_{12}, \lambda_{23})$, and let 
$$\Theta = \left\{ (\omega_{11}, \omega_{22}, \omega_{23}, \omega_{33}, \lambda_{12}, \lambda_{23}) \in \rr^6 : \right. $$ 
$$\left. \omega_{11} > 0 , \omega_{22} > 0 , \omega_{33} > 0 , \omega_{11}\omega_{22} > \omega_{12}^2  \right\}.$$
Let $\bff: \rr^6 \rightarrow \rr^6$ given by
\begin{eqnarray*}
f_{11}(\theta) &  =  & \omega_{11} \\
f_{12}(\theta) &  =  & \omega_{11}\lambda_{12} \\
f_{13}(\theta) &  =  & \omega_{11}\lambda_{12} \lambda_{23} \\
f_{22}(\theta) &  =  & \omega_{22} + \omega_{11}\lambda_{12}^2 \\
f_{23}(\theta) &  =  & \omega_{22}\lambda_{23} + \omega_{11}\lambda_{12}^2 \lambda_{23} + \omega_{23}\\
f_{33}(\theta) & =  &  \omega_{33} + \omega_{22} \lambda_{23}^2 + \omega_{23}\lambda_{23} + \omega_{11}\lambda_{12}^2\lambda_{23}^2.
\end{eqnarray*}
Since $\omega_{11} > 0$ in $\Theta$, $\lambda_{12}$ is identifiable by the formula $\lambda_{12}  =  \frac{ f_{12}}{f_{11}}$.  On the other hand, the parameter $\lambda_{23}$ is generically identifiable by the formula $\lambda_{23} = \frac{ f_{13}}{f_{12}}$.  But this formula does not show that $\lambda_{23}$ is identifiable, because $f_{12} = 0$ for some values in $\bff(\Theta)$ in particular, whenever $\lambda_{12} = 0$.  In fact, it is not difficult to show that $\lambda_{23}$ cannot be determined whenever $\lambda_{12} = 0$, and hence this parameter is not identifiable. 
\end{ex}

To describe the setup for determining identifiability with computational algebra, we need to introduce the closely related notion of constraints.

Let $\rr[\bfp] := \rr[p_1, \ldots, p_n]$ be the polynomial ring in indeterminates $p_1, \ldots, p_n$.  The vanishing ideal (or constraint set) of $S \subseteq \rr^n$ is the set
$$
\ci(S) :=  \{  g \in \rr[\bfp] :  g(\bfa) = 0  \,\, \mbox{ for all  } \bfa \in S \}.
$$
The vanishing ideal, as the name implies, is an ideal in the ring $\rr[\bfp]$, that is, it is closed under addition and under multiplication by an arbitrary polynomial.  The simplest example of an ideal is the ideal generated by a collection of polynomials:
$$
\langle g_1, \ldots, g_r \rangle :=  \{ \sum_{i = 1}^r h_i g_i :  h_i \in \rr[\bfp] \}.
$$
Hilbert's basis theorem says that every ideal has a finite generating set; that is, every ideal can be written as the ideal generated by a finite collection of polynomials.  One of the advantages of working in the language of ideals and generating sets is that they allow for the computation of constraint sets.

\begin{prop}\label{vanishing-ideal}
Let $\bff : \Theta \subseteq \rr^d \rightarrow \rr^n$ be a polynomial parametrization from a full dimensional parameter space.  Then
$$
\ci(\bff(\Theta)) = \langle p_1 - f_1(\bft), \ldots, p_n - f_n(\bft) \rangle \cap \rr[\bfp].
$$
In particular, the constraints can be determined by eliminating the $t$-indeterminates.
\end{prop}

The intersection in Proposition \ref{vanishing-ideal} can be computed using Gr\"obner bases with elimination term orders, see below.

Constraint sets can also be used to determine whether or not parameters are identifiable.  Indeed, consider the modified parametrization map $\tilde{\bff}  = (s, f_1, \ldots, f_n)^T :  \Theta \rightarrow \rr^{n+1}$.  Let $\rr[q,\bfp]$ be the polynomial ring with one extra indeterminate corresponding to the parameter function $s$.  Let $\ci(\tilde{\bff}(\Theta))$ be the vanishing ideal of the image.  Then we have the following proposition.

\begin{prop} \label{prop:identpoly}
Suppose that $g(q, \bfp) \in \ci(\tilde{\bff}(\Theta))$ is a polynomial such that $q$ appears in this polynomial, $g(q, \bfp) = \sum_{i = 0}^d g_i(\bfp) q^i$ and $g_d(\bfp)$ does not belong to $\ci(\bff(\Theta))$.
\begin{enumerate}
\item  If $g$ is linear in $q$, $g =  g_1(\bfp) q - g_0(\bfp)$ then $s$ is generically identifiable by the formula $s = \frac{g_0(\bfp)}{g_1(\bfp)}$.  If, in addition, $g_1(\bfp) \neq 0$ for $\bfp \in \bff(\Theta)$ then $s$ is identifiable.
\item  If $g$ has higher degree $d$ in $q$, then $s$ may or may not be generically identifiable.  Generically, there are at most $d$ possible choices for the parameter $s(\theta)$ given $\bff(\theta)$.
\item  If no such polynomial $g$ exists then the parameter $s$ is not generically identifiable.
\end{enumerate} 
\end{prop}  

\begin{proof}
The ideal $\ci(\tilde{\bff}(\Theta))$ consists of all polynomials $g(q,\bfp)$ in $q, p_1, \ldots, p_n$ such that $g(s(\theta), \bff(\theta)) = 0$ for all $\theta \in \Theta$.  Suppose that there exists a polynomial $g$ satisfying the conditions of the theorem.  Since $g_d(\bfp)$ does not belong to $\ci(\bff(\Theta))$ there exist a dense open subset $U$ of $\Theta$ such that none of the $g_i(\bff)$ are zero.  Letting $\theta \in U$, we see that $s(\theta)$ is one of the solutions to the nondegenerate equation $g(q, \bff(\theta)) = 0$.  If $g$ is linear in $q$ this equation has a unique solution, and hence $s(\theta)$ is generically identifiable.  If $g_1(\bfp) \neq 0$ for all $\bfp \in \bff(\Theta)$ then we can take $U = \Theta$ and $s(\theta)$ is identifiable.

If $g$ is the lowest degree polynomial in $\ci(\tilde{\bff}(\Theta))$ satisfying the desired properties, and $g$ is not linear, then $s(\theta)$ will be one of the $d$ complex solutions to the equation $g(q,\bff(\theta)) = 0$.  This may or may not (generically) identify $s(\theta)$, depending on additional constraints on $\Theta$.  For instance, if $g(q, \bfp)$ has the form 
$$g(q, \bfp)  =  g_d(\bfp) q^d -  g_0(\bfp)$$
and we know that $s(\theta)> 0$, then $s(\theta)$ will be generically identified.  On the other hand, we will see an example in the next section due to Brito, where $d = 2$ and the parameter is not identified.

On the other hand, suppose that $s(\theta)$ is generically identifiable on $U \subseteq \Theta$.  Let $\bfp \in \bff(U)$.  Let $I \subseteq \rr[q]$ be the ideal generated by the evaluations of all polynomials $g \in \ci(\tilde{\bff}(\Theta))$ at the point $\bfp$.  Since $\rr[q]$ is a principal ideal domain, it has a single generator.  If this generator is the zero polynomial, then a priori, every value of $s(\theta)$ in $s(\Theta)$ is compatible with $\bfp$, but this contradicts identifiability because $s$ is not constant, $s(\Theta)$ has more than one point.  This implies that $I$ contains a nonzero polynomial.  This implies that $\ci(\tilde{\bff}(\Theta))$ must have contained a polynomial $g$ with nonzero degree $d$ in $q$ with leading coefficient $g_{d}(\bfp) \notin \ci(\bff(\Theta))$, for if the leading coefficient of every polynomial in $\ci(\tilde{\bff}(\Theta))$ is in $\ci(\bff(\Theta))$ then every coefficient of every polynomial in  $\ci(\tilde{\bff}(\Theta))$ is in $\ci(\bff(\Theta))$.  This implies that $I = \langle 0 \rangle$, which is a contradiction.
\end{proof}

If the integer $d>0$ is the lowest nonzero degree in $q$ of any polynomial in $\ci(\tilde{\bff}(\Theta))$, then there are $d$ complex values for the parameter $s(\theta)$ that are compatible with $\bff(\theta)$.  We call such a parameter \emph{algebraically $d$-identifiable}.  As mentioned in Proposition \ref{prop:identpoly}, a parameter that is algebraically $d$-identifiable may or might not be identifiable.  A parameter is \emph{$d$-identifiable} if there are $d$ different $\theta' \in \Theta$ such that $\bff(\theta')  = \bff(\theta)$ and $s(\theta')$ are all distinct.  See \citep{Allman2009} for an example of a model in phylogenetics that is algebraically $12$-identifiable but is (conjecturally) only $8$-identifiable, at worst.

The existence or nonexistence of a polynomial $g$ satisfying the conditions of Proposition \ref{prop:identpoly} can be decided by a Gr\"obner basis computation, which we now explain.  Basic information about Gr\"obner basis can be found in \citep*{Cox07}.

A term order $\prec$ on the polynomial ring $\rr[\bfp]$ is a total ordering on the monomials in $\rr[\bfp]$ that is compatible with multiplication and such that $1$ is the smallest monomial; that is, $1 = \bfp^{\bf 0} \preceq \bfp^\bfu$ for all $\bfu \in \nn^n$ and if $\bfp^\bfu \preceq  \bfp^\bfv$ then $\bfp^\bfw \cdot \bfp^\bfu \preceq  \bfp^\bfw\cdot \bfp^\bfv$.  Since $\prec$ is a total ordering, every polynomial $g \in \rr[\bfp]$ has a well-defined largest monomial.  Let ${\rm in}_\prec(g)$ be the largest monomial appearing in $g$.  For an ideal $I \subseteq \rr[\bfp]$ let ${\rm in}_\prec(I) =  \langle {\rm in}_\prec(g) :  g \in I \rangle$.  This is called the initial ideal of $I$.  A finite subset $\cg \subseteq I$ is called a Gr\"obner basis for $I$ with respect to the term order $\prec$ if ${\rm in}_\prec(I)  =  \langle {\rm in}_\prec(g) : g \in \cg \rangle$.  The Gr\"obner basis is called reduced if the coefficient of ${\rm in}_\prec(g)$ in $g$ is one for all $g$, each ${\rm in}_\prec(g)$ is a minimal generator of ${\rm in}_\prec(I)$, and no terms besides the initial terms of $\cg$ belong to ${\rm in}_\prec(I)$.  For a fixed ideal $I$ and term order $\prec$, the reduced Gr\"obner basis of $I$ with respect to $\prec$ is uniquely determined.  Note, however, that as the term order varies, the reduced Gr\"obner basis of $I$ will also change.

Among the most important term orders is the lexicographic term order, which can be defined for any permutation of the variables.  In the lexicographic term order we declare $\bfp^\bfu \prec \bfp^\bfv$ if and only if the left most nonzero entry of $\bfv - \bfu$ is positive.  Stated colloquially, this is the term order that makes $p_1$ so expensive its degree dominates the term order.  If two monomials have the same degree in $p_1$, then we compare the degrees of $p_2$, and so on.  Generalizing the lexicographic order are the elimination orders.  These are obtained by splitting the variables into a partition $A \cup B$.  In the elimination order $\bfp^\bfu \prec \bfp^\bfv$ if $\bfp^\bfv$ has larger degree in the $A$ variables than $\bfp^\bfu$.  If $\bfp^\bfv$ and $\bfp^\bfu$ have the same degree in the $A$ variables, then some other term order is used to break ties.

Computation of Gr\"obner bases is via Buchberger's algorithm.  A key ingredient to this algorithm is the division algorithm of multivariate polynomials.  Once a term ordering is fixed, the division algorithm of multivariate polynomials consists in canceling \emph{leading terms} until no term in the remainder can be divided by the leading term of the divisor much in the same way as the familiar division algorithm for univariate polynomials.  The following proposition provides an algebraic procedure for
the identification of parameters. 

\begin{prop}
Let $\prec$ be an elimination term order with respect to the partition $\{q\} \cup \{p_1, \ldots, p_n\}$.  Let $\cg = \{g_1, \ldots, g_n \}$ be a reduced Gr\"obner basis for $\ci(\tilde{\bff}(\Theta))$ with respect to the term order $\prec$.  The Gr\"obner basis $\cg$ contains polynomials of the lowest nonzero degree of the form from Proposition \ref{prop:identpoly}, if such a polynomial exists.  In particular, if no polynomial of $\cg$ contains the indeterminate $q$, then $s$ is not generically identifiable.
\end{prop}

\begin{proof}
 Let $\cg = \{g_1, \ldots, g_n \}$ be a reduced Gr\"obner basis for $\ci(\tilde{\bff}(\Theta))$ with respect to the elimination order $\prec$.  By virtue of being an elimination order, this set also contains a reduced Gr\"obner basis for $\ci(\bff(\Theta))$.  Let $\cg'$ denote this reduced Gr\"obner basis.  Note that $\cg'  = \cg \cap \rr[\bfp]$ by properties of elimination orders.
 
 Now, let $g$ be a polynomial satisfying the conditions of Proposition \ref{prop:identpoly} of the lowest nonzero degree in $q$.  We can apply the division algorithm by  $\cg'$ to $g$ to get a remainder $\tilde{g}$.  Since the leading coefficient of $g$ in $q$ is not in $\ci(\bff(\Theta))$, this leading coefficient does not reduce to zero by division by $\cg'$.  Thus, $\tilde{g}$ has the same degree as $g$ in $q$.  Now ${\rm in}_\prec(\tilde{g}) \in {\rm in}_\prec(\ci(\tilde{\bff}(\Theta)))$, thus ${\rm in}_\prec(\tilde{g})$ is divisible by some leading term of some polynomial in $\cg$.  But by our reduction assumption it is not divisible by the leading monomial of any element of $\cg'$.  Hence, it must be divisible by some element of $h \in \cg$ whose leading term has a nonzero power of $q$ in it.  Since $g$ has the lowest possible degree in $q$, then so does $\tilde{g}$, and so must $h$, to divide the leading term of $\tilde{g}$.  
 
On the other hand, if no such polynomial exists, then there could not be any $q$ appearing in any elements of the reduced Gr\"obner basis of $\ci(\tilde{\bff}(\Theta))$.
\end{proof}

\begin{ex}
The following Macaulay2 code computes the unique lowest degree polynomial $g(q, \bfp)$ for the problem of determining the identification of the parameter $\lambda_{23}$.  Instead of using indeterminates $p_1, \ldots, p_6$, we use $s11, \ldots, s33$ to match the $f_{11}, \ldots, f_{33}$ in the parametrization.

\begin{verbatim}
S = QQ[w11,w22,w23,w33,l12,l23];
R = QQ[q,s11,s12,s13,s22,s23,s33, 
  MonomialOrder => Eliminate 1];
f = map(S,R,matrix{{
l23,
w11,
w11*l12,
w11*l12*l23,
w22 + w11*l12^2,
w22*l23 + w11*l12^2*l23 + w23,
w33 + w22*l23^2 + w23*l23 
  + w11*l12^2*l23^2}});
kernel f;
\end{verbatim}

The output is the Gr\"obner basis of the ideal $\ci(\tilde{\bff}(\Theta))$ which consists of a single polynomial  $s_{12} q - s_{13}$.
\end{ex}


\section{GAUSSIAN STRUCTURAL EQUATION MODELS}

Let $G = (V, D, B)$ be a graph with vertex set $V$, a set of directed edges $D$, and a set of bidirected edges $B$.  We assume the $V = \{1,2,\ldots, m \}$ and that the subgraph of directed edges is acyclic and topologically ordered (that is, $i \to j \in D$ implies that $i < j$).  Let $PD_n$ denote the set of $m \times m$ symmetric positive definite matrices.  Let $PD(B) :=  \{ \Omega \in PD_m :  \omega_{ij} = 0  \mbox{ if }i \neq j \mbox{ and } i \bi j  \notin B \}$.

The Gaussian structural equation for the graph $G$ is a set of linear relationships between random variables $X_i : i \in V$ induced by the graph and starting with correlated noise terms.  In particular, let $\epsilon$ be a centered $n$-dimensional jointly normal random vector $\epsilon  \sim  \mathcal{N}(0, \Omega)$ such that $\Omega \in PD(B)$.  For each $i \to j \in D$ let $\lambda_{ij}\in \rr$ be a parameter.  For each $j \in V$ define
$$ X_j =  \sum_{i: i \to j \in D} \lambda_{ij} X_i  +  \epsilon_j.$$
The random vector $X$ has a jointly normal distribution with $X \sim \cn(0, \Sigma)$ where
$$
\Sigma  =  (I - \Lambda)^{-T} \Omega (I - \Lambda)^{-1}
$$
where $\Lambda$ is the strictly upper triangular matrix with $\Lambda_{ij} = \lambda_{ij}$ if $i \to j \in D$ and $\Lambda_{ij} = 0$ otherwise.

Said in the language of statistical models and mappings between parameter spaces in the previous section, the Gaussian structural equation model is a map that associates to a parameter vector $(\Lambda, \Omega) \in \rr^{\#D} \times PD(B)$ the normal distribution $\cn(0, \Sigma)$.  

The parameters of most frequent interest in structural equations models are the entries of $\Lambda$ and the entries of $(I - \Lambda)^{-1}$.  The parameter $\lambda_{ij}$ is called the \emph{direct causal effect} of $X_i$ on $X_j$.  The parameter $(I - \Lambda)^{-1}_{ij}$ is called the \emph{total effect} of $X_i$ on $X_j$.  Parameter identification in these structural equation models asks for formulas to recover the direct and total effects given the covariance matrix $\Sigma$.

Note that by basic properties of algebraic graph theory, the entries in $(I - \Lambda)^{-1}$ have a combinatorial interpretation in terms of directed paths in the graph $G$.  A directed path from $i$ to $j$ is a sequence of directed edges $i \to i_1 \to i_2 \to \cdots \to j$.  The set of all paths from $i$ to $j$ is denoted $\cp(i,j)$.  Then
$$(I - \Lambda)^{-1}_{ij}  =  \sum_{P \in \cp(i,j)}  \prod_{k \to l \in P}  \lambda_{kl}.$$
Another type of parameter of occasional interest are the \emph{path specific effects} which come from the monomials $\prod_{k \to l \in P}  \lambda_{kl}$ for $P \in \cp(i,j)$.

If all the $\lambda_{ij}$ parameters are (generically) identifiable, then so too are the entries of $\Omega$.  Indeed, once one can determine all the $\lambda_{ij}$ we can simply determine $\Omega$ by

\begin{equation}\label{omega}
\Omega =  (I - \Lambda)^T \Sigma (I - \Lambda).
\end{equation}

This fact provides another reason why much of the attention is focused on only parameters that involve $\Lambda$ when studying identifiability problems for Gaussian graphical models.

Note that since both the direct effects and total effects are polynomial, we can employ the techniques from the previous section to decide whether or not parameters in the model are identifiable.

\begin{ex}
Consider the graph $G$ in Figure 1.

\begin{figure}[h]\label{fig:instr}
\begin{center}
\resizebox{!}{2cm}{\includegraphics{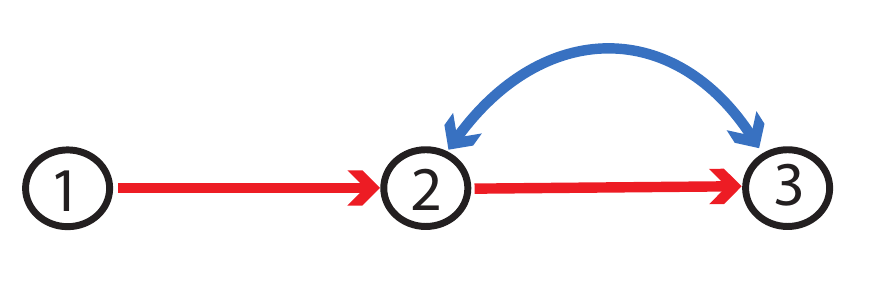}}
\end{center}
\caption{Instrumental Variable}
\end{figure}

This model is traditionally called an instrumental variables model (random variable $X_1$ is the instrument).  The formulas for $\sigma_{ij}$ in terms of $\Lambda$ and $\Omega$ are obtained from  the factorization of $\Sigma = (I - \Lambda)^{-T} \Omega (I - \Lambda)^{-1}$, are simply the $f_{ij}$ from Example \ref{ex:instrument}.  This model is generically identifiable but not identifiable, since it is not possible to identify $\lambda_{23}$ when $\lambda_{12} = 0$.
\end{ex}

\begin{ex}\label{Brito4nodes}
Consider the graph $G$ in Figure 2.
\begin{figure}[h]
\begin{center}
\resizebox{!}{2cm}{\includegraphics{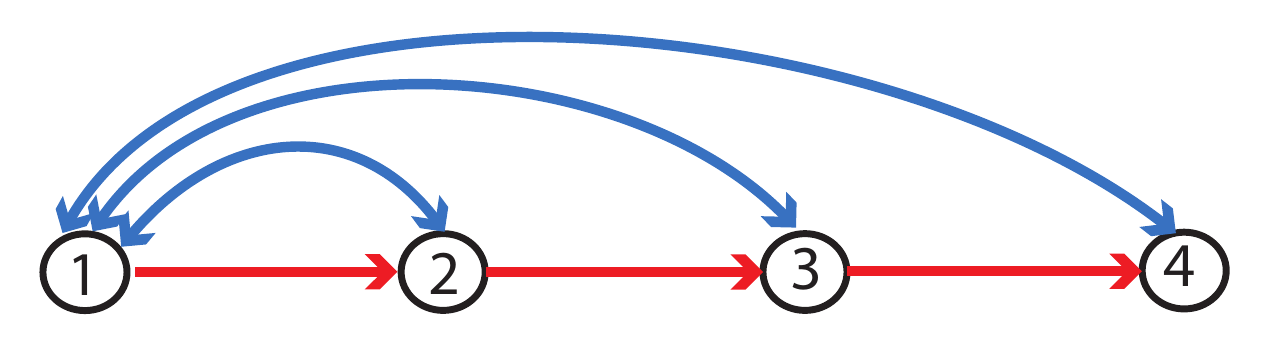}}
\end{center}
\caption{A $2$-identified SEM}
\end{figure}
This example is originally due to Brito and it shows that the ``GAV-criterion'' described in \citep{Brito02b} is not complete.  Using the ideas from the previous section, we compute the elimination polynomials to deduce that every parameter except $\omega_{11}$ is the solution to a quadratic polynomial with coefficients determined by $\Sigma$.  For example, $\lambda_{23}$ is the solution to the quadratic equation:

\small

$$
\begin{array}{lll}
0 & = &
\big(\sigma_{14}\sigma_{22}\sigma_{23}-\sigma_{13}\sigma_{22}\sigma_{24}\big)q^2
 + \big(\sigma_{13}\sigma_{23}\sigma_{24}\\ 
 &  & -\sigma_{14}\sigma_{22}\sigma_{33} -\sigma_{14}\sigma_{23}^2 +\sigma_{12}\sigma_{24}\sigma_{33}+\sigma_{13}\sigma_{22}\sigma_{34} \\ 
 & &   -\sigma_{12}\sigma_{23}\sigma_{34}\big)q + \big(\sigma_{14}\sigma_{23}\sigma_{33}-\sigma_{13}\sigma_{24}\sigma_{33}\big).
 \end{array}
$$

\normalsize

Consider the polynomials for $\lambda_{ij}$.  Since we know there is one real solution for $\Sigma \in \bff(\Theta)$, both solutions for $\lambda_{ij}$ are real.  This means that there are exactly two real lambda matrices compatible with a given $\Sigma$.  Solving for $\Omega$ as
$$
\Omega =  (I - \Lambda)^T \Sigma (I - \Lambda).
$$
with both real choices of $\Lambda$ give two real possibilities for $\Omega$, which hence must be the real roots of the identification polynomials for $\Omega$.  Since $\Sigma$ is positive definite, so must be $\Omega$ in both these cases.  This implies that in this case $\bff$ is a generically 2-to-1 map, that is, the model is $2$-identified.
\end{ex}

\section{COMPUTATIONAL RESULTS}

In this section we present our computational results on the identification of all Gaussian structural equation models on three and four random variables. Using the ideas presented in Section \ref{section:comp-alg}, we computed all generically identifiable parameters for each of the $2^{6} = 64$ models on three variables and each of the $2^{12} = 4096$ models on four variables. 
The parameters identified are the \emph{direct causal effects} (the entries of $\Lambda$), the \emph{total effects} (the entries of $(I-\Lambda)^{-1}$), the \emph{path specific effects}, and the entries of $\Omega$.  The results of these computations are displayed on our project website:

\begin{center}
{\tt http://graphicalmodels.info/}
\end{center}

The website contains all identifiability results in formatted text. It also displays colored pictures encoding identifiability of parameters in $\Lambda$ and $\Omega$. The parameter $\omega_{ii}$ in $\Omega$ is represented by the node labeled $X_{i}$ in the colored graph. An edge or a node is colored green if the associated parameter is generically identified, it is colored blue it is is algebraically $k$-identifiable with $k \geq 2$, otherwise it is colored red.  On a black and white print-out, a green edge is recognized by a circle surrounding the label of the corresponding parameter, a blue edge by an ellipse and a red edge by a rectangle.  For example the graphical model in Figure \ref{web-graph} has generically identified direct effect $\lambda_{23}$ and generically identified  $\Omega$ parameters $\omega_{11}$, $\omega_{24}$ and $\omega_{33}$. 

The colored graph does not encode total effects or path-specific effects, but in this example, the total effect of $X_{2}$ on $X_{4}$, given by the polynomial $\lambda_{23}\lambda_{34}+\lambda_{24}$,  is generically identified as the solution to the equation $$\sigma_{12}q-\sigma_{14} = 0$$ but does not satisfy the back-door criterion. No other total or path-specific effect in this graph is identified.

\begin{figure}[h]
\begin{center}
\includegraphics[totalheight=7cm]{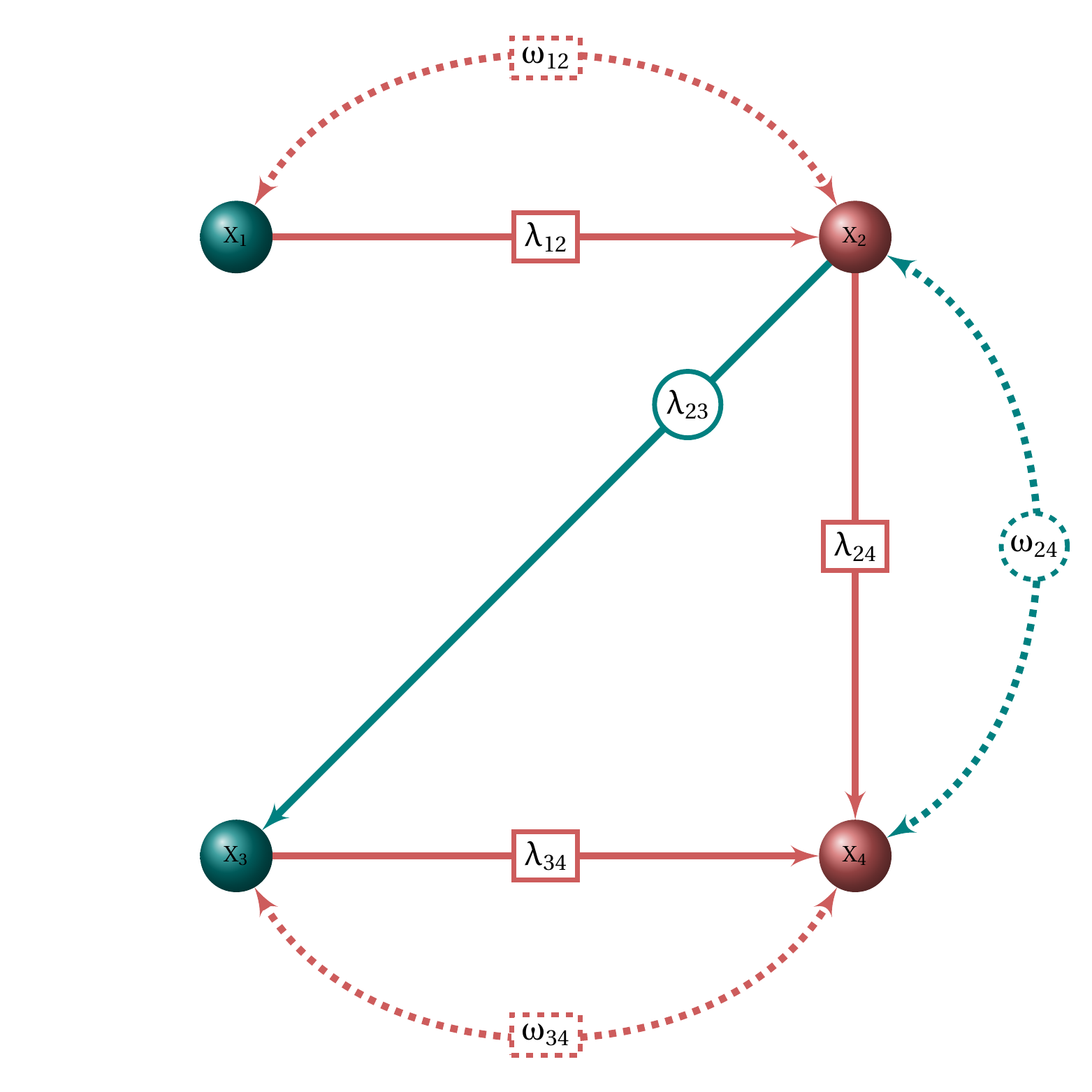}
\end{center}
\caption{Four variable SEM model with total effect of $X_{2}$ on $X_{4}$ generically identifiable.}\label{web-graph}
\end{figure}

Besides the inherent usefulness of solving the (generic) identifiability problem for all models on three and four variables, our main motivation for creating this database is that it can be used to test the efficacy and correctness of current and future graphical criteria for identifiability.  For this reason, we have also computed which direct causal effects are identified by the single-door criterion or instrumental variables, and which total effects are identified by the back-door criterion.  

We have developed a Singular \citep{GPS09} library to perform all the previous computations. The library requires the latest version of Singular.  Its graphing capabilities require some special \LaTeX \, packages and a Mac OS X environment. The library, its documentation and installation instructions can be found on the website. Currently we are also porting this library to Macaulay2 \citep{M2}. Future plans include extending this library to include more graphical criteria.

In the next subsections, we summarize the results of these computations for $3$ and $4$ random variables.

\subsection{THREE RANDOM VARIABLES}

\begin{thm}
Of the $64$ graphs on three vertices,
\begin{itemize}
\item[(i)] there are exactly 31 graphs that are generically identifiable and 33 graphs that are not generically identifiable.

\item[(ii)] The single-door criterion and instrumental variables form a \emph{complete} method to generically identify direct causal effects for SEM models on three variables.
\end{itemize}
\end{thm}

\begin{proof}
There are 27 \emph{bow-free models} on three variables, that is, models satisfying the condition that the errors for variables $i$ and $j$ are uncorrelated if variable $i$ occurs in the structural equation for variable $j$. \citet{Brito04} shows that every bow-free model is generically identified.   Our computations show that all direct causal effect parameters in a bow-free model on three variables are generically identified by the single-door criterion.

Table \ref{IV} lists the four remaining generically identifiable models. Each of these graphs has exactly one direct causal effect parameter which is identified by an instrumental variable but not by the single-door criterion.

\begin{table}[h] 
\caption{SEMs on three variables with one generically identified direct causal effect by an instrumental variable.} 
\label{IV} 
\begin{center} 
\begin{tabular}{ll} 
\multicolumn{1}{c}{\bf Directed edges}  &\multicolumn{1}{c}{\bf Bidirected edges} \\ 
\hline \\ 
$2 \to 3$ & $1 \bi 2,\ 2 \bi 3$ \\  
$1 \to 3$ & $1 \bi 2,\ 1 \bi 3$ \\ 
$1 \to 2$ & $1 \bi 2,\ 1 \bi 3$ \\ 
$1 \to 2,\ 2 \to 3$ & $2 \bi 3$ \\ 
\end{tabular} 
\end{center} 
\end{table} 

The remaining 33 graphs are not generically identifiable. Nevertheless, 13 of these non-identifiable graphs contain at least one identified direct causal effect.  
\end{proof}




The previous theorem describes the identification of direct causal effects. Nevertheless, if a model is not identifiable, Equation \ref{omega} cannot be used to identify the parameters in $\Omega$.  To our knowledge, there are no graphical methods to identify these parameters. So even for this simple model our algebraic approach provides new insights. For example, in the model with directed edges $1 \to 2,\ 2 \to 3$ and bidirected edges $1\bi 2,\ 2 \bi 3$, the parameter $\lambda_{23}$ is generically identified by an instrumental variable but $\lambda_{12}$ is not generically identified. The parameter $\omega_{11}$ is identifiable and $\omega_{23}$, and $\omega_{33}$ are generically identified. The parameter $\omega_{23}$ is generically identifiable by the formula:

$$ \omega_{23} =  \frac{\sigma_{12}\sigma_{23} - \sigma_{13}\sigma_{22}}{\sigma_{12}}.$$


\subsection{FOUR RANDOM VARIABLES}

The case of models on four random variables gets significantly harder. First of all, there are $2^{12} = 4096$ graphs, some of them are not generically identifiable but $2$-identifiable.  In these cases, no graphical criterion can identify these parameters. Example \ref{Brito4nodes} exhibits this behaviour.  Moreover, even when the models are generically identifiable or just some subset of parameters are generically identified, the featured graphical criteria are no longer complete, i.e., there are several SEM models on four variables where the algebraic method is the only tested approach capable of (generically) identifying certain parameters.  Table \ref{identified-by-algebra} lists some examples with this behaviour. It remains to be seen whether the same statement holds when we include even more of the existing graphical criteria. 

\begin{table}[h] 
\caption{Three SEMs on four variables with two generically identified parameters via algebraic methods.}
\label{identified-by-algebra} 
\begin{center} 
\begin{tabular}{ll} 
\multicolumn{1}{c}{\bf Directed edges}  &\multicolumn{1}{c}{\bf Bidirected edges}  \\ 
\hline \\ 
$1 \to 3, 2 \to 4, 3 \to 4$ & $1 \bi 2, 1 \bi 3, 1 \bi 4$ \\
$1 \to 2, 2 \to 3, 2 \to 4$ & $1 \bi 2, 1 \bi 3, 2 \bi 4$ \\ 
$1 \to 2, 1 \to 3, 2 \to 3, 3 \to 4$ & $2 \bi 3, 2 \bi 4$\\ 
\end{tabular} 
\end{center} 
\end{table} 

While most of the models took seconds to be computed, there were a handful of graphs ($< 100$) that took weeks or even months to be identified. While large numbers of edges and  bows is expected in graphs with this anomalous behavior, no other apparent combinatorial description was easily identified. For example, in the model with directed edges $1 \to 2,\ 1 \to 4,\ 2 \to 3,\ 3 \to 4$ and bidirected edges $1 \bi 2,\ 1 \bi 3,\ 1 \bi 4$ the computation to show that $\omega_{44}$  is not identifiable took more that 75 days. 

The following theorem summarizes our findings.

\begin{thm}
Of the $4096$ graphs on four variables
\begin{itemize}
\item[(i)] exactly $1246$ are generically identifiable, $6$ are algebraically $2$-identified, and $2844$ are not generically identifiable.

\item[(ii)] Of the $1246$ generically identifiable models, exactly $1093$ are generically identified by the single-door and instrumental variables criteria and the remaining $153$  generically identified models contain direct causal effect parameters only identified by the algebraic method.

\item[(iii)] There are exactly $729$ bow-free models, each generically identified by the single-door criterion. 
\end{itemize}
\end{thm}   

Table \ref{2identified} lists the $6$ algebraically $2$-identified SEM models on four variables and the time (in seconds) to perform all computations.

\begin{table}[h] 
\caption{Algebraically $2$-identified SEMs on four variables.}
\label{2identified} 
\begin{center} 
\begin{tabular}{llc} 
\multicolumn{1}{c}{\bf Directed edges}  &\multicolumn{1}{c}{\bf Bidirected edges}   &\multicolumn{1}{c}{\bf Time}   \\ 
\hline \\ 
$1 \to 2, 2 \to 3, 3 \to 4$ & $1 \bi 2, 1 \bi 3, 1 \bi 4$ & 4.5 \\ 
$1 \to 2, 2 \to 3, 2 \to 4$ & $1 \bi 2, 1 \bi 3, 1 \bi 4$ & 0.7\\ 
$1 \to 2, 1 \to 4, 2 \to 3$ & $1 \bi 2, 1 \bi 3, 1 \bi 4$ & 0.6\\ 
$1 \to 2, 1 \to 3, 3 \to 4$ & $1 \bi 2, 1 \bi 3, 1 \bi 4$ & 1.1\\ 
$1 \to 2, 1 \to 3, 2 \to 4$ & $1 \bi 2, 1 \bi 3, 1 \bi 4$ & 0.9\\ 
$1 \to 2, 1 \to 3, 1 \to 4$ & $1 \bi 2, 1 \bi 3, 1 \bi 4$ & 0.3\\ 
\end{tabular} 
\end{center} 
\end{table} 

The first SEM model in Table \ref{2identified} corresponds to Example \ref{Brito4nodes}. Each of the remaining five models exhibit a similar behavior, $\omega_{11}$ is identified and 
all direct causal effects  are the solution to a  quadratic polynomial with coefficients
determined by $\Sigma$. Nonetheless, in the second model the parameters $\omega_{33}$ and $\omega_{44}$ are generically identifiable and in the last model 
the parameters $\omega_{22}, \omega_{33}$ and  $\omega_{44}$ are generically identifiable.  



\subsection{CONSTRAINT SETS}

As described in Section \ref{section:comp-alg}, the identification of a particular parameter in a model parametrized by $\bff$ requires the computation of the vanishing ideal $\ci(\bff(\Theta))$, or constraint set. 
Our website also displays the results of these vanishing ideal computations, and determines when the ideals $\ci(\bff(\Theta))$ are generated by determinantal constraints (generalizations of conditional independence constraints), using \emph{trek separation} \citep{Sullivant08}.

\begin{thm}
The vanishing ideal of any structural equation model on three variables is determinantal.  Of the $4096$ structural equation model  on four random variables, the vanishing ideals of exactly $33$ are not determinantal.
\end{thm}

\section{DISCUSSION}

We have described a framework using computational algebra for determining whether or not parameters are identifiable in statistical models.  We used our framework to provide the first systematic study of the identifiability problem for structural equation models.  We have displayed the results of our computations for graphs with three or four random variables on a searchable website.  Developing a general characterization of which parameters for which graphs are, in fact, identifiable remains a major open problem in the theoretical study of structural equation models.

One observation arising from our large scale computational study is that two graphs which combinatorially seem very similar might have drastically different running times when it comes to verifying if the parameters in the model are identifiable.  The longest computations seemed to occur in the proofs of \emph{non}identifiability of some of the parameters.  This phenomenon deserves more careful study.  In particular, we need to address the question of whether or not this has to do with our implementation (for example, if changing the term order might speed up computations) or if some graphs are simply intrinsically more difficult to prove or disprove identifiability.

\subsubsection*{References} 
 
\renewcommand\refname{}
\bibliographystyle{agms}

\end{document}